\documentclass[runningheads]{llncs}
\usepackage[T1]{fontenc}
\usepackage{amssymb}
\usepackage{graphicx}
\usepackage[all]{xy}
\usepackage{amsmath}

\newtheorem{thm}{Theorem}

\begin{document}
\title{Image Recognition via Vaisman--Neifeld's Geometry}
%
%
\author{No\'emie C. Combe\inst{1}\orcidID{0000-0003-4540-7376} \and
Hanna K. Nencka }
\authorrunning{N. Combe et al.}
%
\institute{University of Warsaw, Ul. Banacha 2, 02-097 Warsaw, Poland 
\email{n.combe@uw.edu.pl}\\
\url{https://noemie-combe-23.webself.net} \\}

\maketitle              
\begin{abstract}
We introduce a new approach to the reconstruction of hidden structures from incomplete data, unifying techniques from geometric integration and topological analysis within the pioneering frameworks of Vaisman and Neifeld. Our method transcends traditional iterative schemes by employing a refined geometric decomposition of configuration spaces into invariant foliations and moment maps, thereby resolving the intrinsic ambiguities of underdetermined inverse problems. By synergistically combining Vaisman’s deep insights into symmetry and Neifeld’s analytic methodologies, we establish a robust, noise-resistant paradigm that not only ensures computational tractability but also fundamentally redefines the landscape of reconstruction in imaging and structural analysis. This framework paves the way for transformative applications across diverse scientific domains, heralding a new era in the synthesis of geometry and topology for inverse problem solving.
\end{abstract}

\section{Introduction}
Reconstruction problems stand among the biggest challenges in mathematics and applied sciences: \emph{How does one rebuild a hidden structure from fragmented, incomplete, or distorted data?} Whether it is deducing the three-dimensional conformation of a protein from blurry two-dimensional microscope images \cite{CSSS}, or inferring the structure of a graph from partially observed subgraphs, these inverse problems lie at the core of diverse disciplines—ranging from cryo-electron microscopy and  tomography to astronomy, radio-astronomy, medicine and graph theory as well as modern machine learning.

The intrinsic difficulty of reconstruction is underscored by the fact that infinitely many structures may be consistent with the same incomplete data. In practice, traditional methods, such as iterative algorithms, are hampered by three fundamental challenges:
\begin{enumerate}
    \item \textbf{Uniqueness:} Does the available data suffice to uniquely determine the original structure?
    \item \textbf{Noise:} How can one mitigate the deleterious effects of noise and measurement imperfections inherent in real-world data?
    \item \textbf{Complexity:} How does one efficiently navigate the vast, often combinatorially complex, space of potential solutions?
\end{enumerate}

In this work, we develop a new framework for reconstruction that both unifies and generalizes previous approaches by harnessing the deep geometric insights of the Vaisman \cite{V} and Neifeld \cite{N} methodologies. Our construction is firmly rooted in the contributions of Gelfand--Goncharov \cite{GG}. 

The {\bf Vaisman framework} interprets reconstruction as a geometric unraveling of symmetry. It resolves the inherent multiplicity of solutions in underdetermined inverse problems by inducing a stratified decomposition of the configuration space into transverse foliations. Each foliation represents an equivalence class of solutions that are indistinguishable under a given projection operator, thereby transforming the ambiguity into a structured hierarchy of invariant submanifolds. The intersection of these orthogonal layers of constraints is isolated, which in turn guarantees unique solutions and robustness against noise. For instance, objects that yield identical two-dimensional images naturally lie on the same leaf of the foliation.

{\it Applications} of our method includes {\bf medicine} (reconstructing tumours from sparse MRI slices with error guarantees); {\bf quantum computing} (inferring quantum states from noisy partial measurements) and  {\bf artificial intelligence} (training generative models to impute missing data through geometric manifold learning and differential topological constraints). 

{\small {\bf Acknowledgements} The authors thank Philippe Combe for comments and discussions on this paper. This research is part of the project No. 2022/47/P/ST1/01177 co-founded by the National Science Centre and the European Union's Horizon 2020 research and innovation program, under the Marie Sklodowska Curie grant agreement No. 945339.For the purpose of Open Access, the author has applied a CC-BY public copyright licence to any Author Accepted Manuscript (AAM) version arising from this submission.
\includegraphics[width=1cm, height=0.5cm]{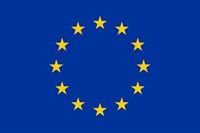}.}

\section{Projections and Inverse Morphisms}
 We consider at least two independent projections of a three-dimensional object onto distinct two-dimensional planes. The existence of such independent projections is necessary to ensure sufficient information for reconstruction. This leads to the problem of defining an inverse function that maps the planar projections back to the original volumetric object.

\begin{itemize}
\item A single projection collapses 3D information onto a plane, losing depth and orientation data. For example, in electron microscopy, a single micrograph of a particle provides no information about its tilt angle relative to the imaging axis.

\item By considering projections along distinct planes, the system gains orthogonal constraints (e.g., tilt angles and rotation axes) that resolve ambiguities. This aligns with Gelfand and Goncharov’s method, which uses statistical properties of projections, such as first moments of plane sections, to derive orientation parameters.
\end{itemize}

\subsection{Gelfand--Goncharov's Viewpoint}  The Radon transform $Rf(\theta,t)$ integrates a function 
$f(x,y,z)$ along planes parameterized by angle 
$\theta$ and offset $t$. Its inverse requires integrating over all possible angles, but practical applications (e.g., cryo-EM) use finite projections. 

 For discrete objects (e.g., particles on a line), the inverse function can be constructed as a linear system where each projection contributes equations. Two independent projections ensure the system is determined (solvable) under non-degenerate conditions \cite{GG}.

\medskip 

 A key idea is to leverage the results of Neifeld on geometric projective 2-dimensional spaces. From these results, it follows that the independent projections induce a pair of independent connections. These connections allow us to extract curvature-related information in two dimensions, including the Riemann tensor and Ricci curvature. Understanding whether these curvature tensors arise from an associative or commutative algebraic structure further informs the reconstruction process.

\, 

Neifeld’s involution principle bridges projective geometry and differential geometry, enabling a dual-connection framework for 3D reconstruction. By interpreting projections as inducing independent connections, this approach generalizes classical integral geometry methods and enhances their robustness, particularly in complex or symmetric settings. Further work could explore applications to quantum state tomography (via projective Hilbert spaces) or algebraic varieties in $\mathbb{CP}^n$.

\subsection{A Generalisation of Goncharov's Reconstruction Method} 
\subsubsection{Goncharov's Construction via the Structure Diagram} We first recall Goncharov's construction, which  discusses families of submanifolds using the language of double fibrations.
Recall that a double fibration is a diagram of manifolds:
\[
\xymatrix{
& A \ar[dl]_{\pi_1} \ar[dr]^{\pi_2} & \\
B & & T
}
\]
where $A$ is a manifold with projections $\pi_1:A\to B$ and $\pi_2:A\to T$ and the combined map 
\[
\pi_1 \times \pi_2 \colon A \to B \times T
\]
is an embedding. For \( x \in B \) and \( t \in T \), define
\[
B_t := \pi_1\bigl(\pi_2^{-1}(t)\bigr), \quad \Gamma_x := \pi_2\bigl(\pi_1^{-1}(x)\bigr).
\]
Thus, the double fibration determines a family \(\{B_t\}\) of submanifolds in \(B\) as well as a dual family \(\{\Gamma_x\}\) of submanifolds in \(T\).
\subsubsection{Geometric realization} Conversely, a family \(\{B_t\}\) of submanifolds in \(B\) defines explicitly the incidence submanifold \( A \subset B \times T \) as:
\[
A:= \{ (x,t) \in B \times T \mid x \in B_t \}.
\]

The projections of \( A \) onto the factors \( B \) and \( T \) then yield a double fibration and the embedded  \( A \subset B \times T \) 
is precisely the incidence submanifold $\{ (x,t) \in B \times T \mid x \in B_t \}$.
~

\begin{remark} 
In integral geometry (e.g., Radon transforms), the incidence submanifold  \( A \subset B \times T \)  is where the kernel of the integral operator is defined. 
\end{remark}
~

Goncharov introduces a {\it Geometric Method for recovering the Mutual Orientation particles
from the Projections} (\cite[Sec. 2]{G}). Using Neifeld and Vaisman's geometrical approach,  we {\it generalize} this below. 

\subsection{Our Generalization of Goncharov's Geometric Method } We start the generalisation by introducing the following proposition. 
\begin{thm}\label{P:1} Let \( O \subset \mathbb{CP}^3 \) be a smooth, non-symmetric 3D object (real 3-dimensional submanifold), and let \( \Pi_1, \Pi_2: \mathbb{CP}^3 \dashrightarrow \mathbb{CP}^2 \) be two independent rational projections onto distinct complex projective planes. Assume:

\begin{enumerate}
\item  {\bf Non-degeneracy of projections}: The restrictions of \( \Pi_1 \) and \( \Pi_2 \) to \( O \) are immersive (the differentials \( d\Pi_i \) have full rank).

\item  {\bf Involution symmetry:} The projections satisfy \(\Pi_2 = \iota \circ \Pi_1 ,\) where \( \iota: \mathbb{CP}^2 \to \mathbb{CP}^2 \) is an anti-holomorphic involution (e.g., a polarity induced by the Fubini-Study metric).

\end{enumerate}

Each projection \( \Pi_i \) defines a holomorphic line bundle \( L_i \to \mathbb{CP}^2 \) equipped with a connection \( \nabla_i \) derived from the Fubini-Study metric. Then, the  involution \( \iota \) induces a duality \[ L_1 \leftrightarrow L_2^* ,\] making \( (\nabla_1, \nabla_2) \) a dual pair.
\end{thm}
We make two remarks, before the proof: 
\begin{enumerate}
\item[(1)] Notice that the {\it immersiveness} of the projection \( \Pi_i \), restricted to the object $O$, for $i\in \{1,2\}$ cannot be assumed unconditionally. It depends on the relationship between the geometry of the object $O$ and the center of projection in $ \mathbb{CP}^3$. Precisely, let $q$ be the center of the projection, chosen generically (i.e. avoiding all tangent planes of $O$). Then  \( \Pi_i \) restricted to $O$ becomes immersive. Conversely, if $q$ lies on a tangent plane of $O$ then the center is no longer generic and the differential  \( \Pi_i \) restricted to $O$ is no longer immersive since $d\Pi_i|_O$ has a kernel.
\item[(2)] The polarity induced by the Fubini-Study metric is a long standing result of classical differential geometry. The Fubini-Study metric defines a Hermitian inner product $\langle -,-\rangle_H$ on $\mathbb{CP}^n$. This allows us to assign to every point $[z]\in \mathbb{CP}^2$ its polar hyperplane {\bf H}$_p$ (given by the set of points orthogonal to $[z]:\, \{[w]\in\mathbb{CP}^2\, |\, \langle z,w\rangle_H=0\})$. These hyperplanes live in the dual projective space $(\mathbb{CP}^2)^\vee$. The Fubini-Study metric provides a canonical identification $ \mathbb{CP}^2\cong (\mathbb{CP}^2)^{\vee}$. This construction defines an anti-holomorphic involution  $\iota$ on $\mathbb{CP}^2$, where the point $p$ is mapped to a point in  $\mathbb{CP}^2$ corresponding to its orthogonal hyperplane {\bf H}$_p$. 
\end{enumerate}
\begin{proof}
1. We first discuss the duality of line bundles.
\smallskip 
Let $\mathcal{O}_{\mathbb{CP}^2}(1)$ be the hyperplane bundle on $\mathbb{CP}^2$. Its dual is  $\mathcal{O}_{\mathbb{CP}^2}(-1)= \mathcal{O}_{\mathbb{CP}^2}(1)^*$. The involution $\iota$ is induced by  the Fubini-Study polarity, which is defined by a Hermitian form $H$ on $\mathbb{C}^3$. In homogeneous coordinates, for 
$z = [z_0 : z_1 : z_2] \in \mathbb{CP}^2,$
we have
$\iota(z) = H^{-1} \overline{z}^T$, where  $H = \mathrm{diag}(1, 1, 1)$ in the standard case.

The key isomorphism is:
\[
\iota^* \mathcal{O}_{\mathbb{CP}^2}(1) \cong \mathcal{O}_{\mathbb{CP}^2}(-1).
\]

This holds because sections of $\mathcal{O}_{\mathbb{CP}^2}(1)$ are linear forms $s = \sum a_i z_i.$ The pullback section $\iota^*s$ is $s \circ \iota = \sum a_i (H^{-1} \overline{z}^T)_i$. 

Since the Fubini-Study metric identifies 
$\overline{\mathcal{O}_{\mathbb{CP}^2}(1)} \cong \mathcal{O}_{\mathbb{CP}^2}(1)^* \quad \text{via } v \mapsto H(\cdot, v),$
we obtain:
\[
\iota^* \mathcal{O}_{\mathbb{CP}^2}(1) \cong \mathcal{O}_{\mathbb{CP}^2}(-1).
\]

Thus, if $L_1$ is a line bundle on $\mathbb{CP}^2$, and we define $L_2 = \iota^* L_1$, then $L_2 \cong L_1^*.$

\medskip 

2. We now discuss the duality of connections. The Chern connection $\nabla_1$ on $L_1 = \mathcal{O}_{\mathbb{CP}^2}(1)$
is defined by the Fubini-Study metric. Its curvature is proportional to the Kähler form $\omega_{\mathrm{FS}}$. Given a connection $\nabla$ on a vector bundle $E\to M$, its curvature $F_{\nabla}$ is the 2-form-valued endomorphism measuring the failure of $\nabla$ to be flat. Formally:
\[F_{\nabla}(X,Y)=\nabla_X\nabla_Y-\nabla_Y\nabla_X-\nabla_{[X,Y]},\] where $X,Y$ are vector fields on $M$. 

Since $\iota$ is an isometry (i.e., $\iota^* \omega_{\mathrm{FS}} = -\omega_{\mathrm{FS}}$), the pullback connection $\iota^* \nabla_1$ on 
$L_2 = \iota^* L_1$
satisfies:
$F_{\iota^* \nabla_1} = \iota^* F_{\nabla_1} = \iota^*(c \, \omega_{\mathrm{FS}}) = -c \, \omega_{\mathrm{FS}},$
where $c$ is the scaling factor. 
The dual connection $\nabla_1^*$ on $L_1^*$ has curvature:
$F_{\nabla_1^*} = -F_{\nabla_1} = -c \, \omega_{\mathrm{FS}}.$

Thus, we conclude:
\[
F_{\iota^* \nabla_1} = F_{\nabla_1^*}.
\]

Connections with equal curvature on a line bundle are gauge-equivalent. Since both $\iota^* \nabla_1$ and $\nabla_1^*$ are Hermitian (i.e., they preserve the respective Hermitian metrics), they must be equal:

\[
\nabla_2 = \iota^* \nabla_1 = \nabla_1^*.
\]
Therefore, we have shown that $L_2 \cong L_1^*$ and that $\nabla_2=\nabla_1^*$. Thus $(\nabla_1,\nabla_2)$ form a dual pair under the involution $\iota$.
\end{proof}

\subsection{Centroids}

The \emph{centroid} (or first moment) of a geometric shape is its ``average position'' in space, computed as the arithmetic mean of all points in the object. Specifically, when an optical device, such as a telescope or microscope, captures an image, it effectively applies a projection operator that maps a three-dimensional object onto a two-dimensional plane, resulting in data that represents the light intensity distribution across that plane. The centroid of this projection is defined as the ``center of mass'' of the light distribution, computed as the weighted average of the pixel positions, and it encapsulates essential spatial information about the object’s location. This concept is crucial in modern technical devices, as accurately determining the centroid facilitates reconstruction, alignment, and tracking of objects from their projected images, thereby enhancing image processing and analysis.

For a projection 
\[
\Pi_i: \mathbb{R}^3 \to \mathbb{R}^2,
\]
the centroid \(\bar{p}_i \in \mathbb{R}^2\) is given by
\[
\bar{p}_i = \left( \frac{1}{N}\sum_{k=1}^{N} x_k, \; \frac{1}{N}\sum_{k=1}^{N} y_k \right),
\]
where \((x_k,y_k)\) are the coordinates of the projected points and \(N\) is the total number of points.

The centroid encodes the translational symmetry of the projected data. For example, shifting the 3D object in space shifts the centroid linearly in the projection.

\subsubsection{Moment Maps}
A \emph{moment map} generalizes the concept of centroids to algebraic/geometric settings, often encoding symmetry-invariant properties of an object.

The moment map 
\[
\mu_i: O \to \mathbb{C}^2
\]
assigns to the 3D object \(O\) the centroid of its projection onto the plane \(\Pi_i\). If \(O\) has a density distribution, \(\mu_i\) computes the first statistical moment (mean) of the projection. For a line or curve, \(\mu_i\) corresponds to the centroid of its projected trace.

Each \(\mu_i\) provides a linear constraint on the orientation of \(O\). Combining \(\mu_1\) and \(\mu_2\) (from two distinct projections) resolves ambiguities in the 3D orientation.

\subsection{Centroid as a Moment Map in Vaisman's Geometric Framework}
In the framework pioneered by Vaisman, the \emph{centroid}—or first moment—of projected data emerges as a fundamental geometric invariant in the reconstruction problem. Specifically, for a three-dimensional object projected onto a two-dimensional plane, the centroid encapsulates the ``average position” of the object's mass distribution. In this setting, the centroids naturally assume the role of moment maps: algebraic invariants that impose linear constraints on the possible orientations and shapes of the original object.

\subsection{Applications and Limitations}

The centroid-based reconstruction framework finds natural applications across a broad spectrum of disciplines:
\begin{itemize}
    \item \textbf{Structural Biology:} In cryo-electron microscopy, the centroids of 2D projections can determine particle orientations, facilitating the accurate reconstruction of protein structures.
    \item \textbf{Network Analysis:} In the study of graphs, classification may be achieved by focusing on zero-dimensional topological features (i.e., connected components), thus avoiding the complexity of higher-dimensional persistent invariants.
   \item \textbf{Medical Imaging:} Techniques for analyzing physiological signals, such as heart rate variability, can benefit from centroid-based approaches that reduce computational overhead without compromising the integrity of the reconstruction.
\end{itemize}

\begin{thm}\label{P:2}
Let \( \mu_i: O \to \mathbb{C}^2 \) be the first moment maps of \( O \) for projections \( \Pi_i \), encoding the centroids of the projected data, and $\nabla_i$ Chern conections on bundles $E_i=\Pi_i^*T\mathbb{CP}^2$ with curvatures $F_{\nabla_i}$.

\smallskip 

Under these conditions:

\begin{itemize}
\item  {\bf Uniqueness}: The orientation of \( O \) (i.e., its position modulo projective transformations) is uniquely determined by the compatibility of \( \nabla_1 \) and \( \nabla_2 \) acting on \( \mu_1 \) and \( \mu_2 \).

\item {\bf Reconstruction}: The original object \( O \) can be reconstructed as the intersection of the parallel transports along \( \nabla_1 \) and \( \nabla_2 \), applied to the moment maps \( \mu_1 \) and \( \mu_2 \).
\end{itemize}
Explicitly, there exists a unique solution \( v \in T\mathbb{CP}^3 \), a {\bf direction vector} (modulo scaling) in the ambiant space, satisfying:

\[
\begin{cases}
\nabla_1 \mu_1 = v \cdot \omega_1, \\
\nabla_2 \mu_2 = v \cdot \omega_2, 
\end{cases}
\]

where \( \omega_i \) are {\bf connection 1-forms} encoding the involution duality \( \omega_1 = \iota^* \omega_2 \);
 $\nabla_i \mu_i$ is the covariant derivative of $\mu_i$, resulting in a $\mathbb{C}^2$-valued 1-form on $O$ \footnote{ The symbol $``\cdot"$ denotes the contraction of the vector $v$ with the $\mathbb{C}^2$-valued  1-form $\omega_i$. Concretely, 
$v \cdot \omega_i$ is the vector in $\mathbb{C}^2$ obtained by evaluating each $\omega_i^j$ on $v$: 

\[ v \cdot \omega_i=\bigg(\begin{smallmatrix} \omega_i^1(v)\\ \omega_i^2(v)\end{smallmatrix}\bigg)\]}. 
\end{thm}

\begin{proof}
\begin{enumerate}
    \item {\it Orientation and Duality of Connections}:
We prove that the involution $\iota$ forces the connections $\nabla_1$ and $\nabla_2$ to be dual, with 
\[
F_{\nabla_1} = -F_{\nabla_2},
\]
resolving orientation ambiguities.

By Theorem \ref{P:1} $\Pi_2 = \iota \circ \Pi_1$, thus the pullback bundles satisfy:
\[
E_2 = \Pi_2^* T \mathbb{CP}^2 = (\iota \circ \Pi_1)^* T \mathbb{CP}^2 = \Pi_1^* \left( \iota^* T \mathbb{CP}^2 \right).
\]
As $\iota$ is anti-holomorphic, we have:
$
\iota^* T \mathbb{CP}^2 \cong \overline{T \mathbb{CP}^2},$
so $E_2 \cong \overline{E_1}.$

\, 

The curvature $F_{\nabla_i}$ is the pullback of the curvature $\Theta$ of the bundle $\left( T \mathbb{CP}^2, g_{FS} \right)$, where $g_{FS}$ is the Fubini-Study metric. As $g_{FS}$ is Kähler-Einstein with Kähler form $\omega_{FS}$ and:
$\iota^* \omega_{FS} = -\omega_{FS},$
it follows that:
\[
\iota^* \Theta = -\Theta.
\]
Thus:
\[
F_{\nabla_2} = \Pi_2^* \Theta = (\iota \circ \Pi_1)^* \Theta = \Pi_1^* \left( \iota^* \Theta \right) = \Pi_1^* ( -\Theta ) = - F_{\nabla_1}.
\]
This antisymmetry fixes the relative signs in parallel transport, resolving orientation ambiguities.

    \item {\it Uniqueness of Solution}:
Since the projections $\Pi_i$ are immersive, the differentials:
\[
d\Pi_i : T_p O \to T_{\Pi_i(p)} \mathbb{CP}^2
\]
are injective. This implies:
\[
\ker(\omega_1) \cap \ker(\omega_2) = \{ 0 \} \subset T_p \mathbb{CP}^3.
\]
Locally, the forms $\omega_1, \omega_2$ form a full-rank system on $T_p O$. We  build a \textbf{linear system for $v\in T_p \mathbb{CP}^3$}:
\[
\begin{cases}
\nabla_1 \mu_1 = v \cdot \omega_1, \\
\nabla_2 \mu_2 = v \cdot \omega_2,
\end{cases}
\]
The above system defines a linear map:
\[
A_p : T_p \mathbb{CP}^3 \to \mathbb{C}^4, \quad A_p(v) = \left( \omega_1(v), \omega_2(v) \right).
\]
The domain of $A_p$ is $\operatorname{dom} A_p = T_p \mathbb{CP}^3$ of real dimension 6 and the codomain is of real dimension $8$. By the immersivity property, $\operatorname{rank}_\mathbb{R} A_p = 6$ (full rank) and so $\ker A_p$ is 0-dimensional. The compatibility condition $\omega_1 = \iota^* \omega_2$ ensures that $\nabla_i \mu_i \in \operatorname{im} A_p$, guaranteeing a unique solution:
\[
v_p \in T_p O \subset T_p \mathbb{CP}^3.
\]
  \item {\it Reconstruction}. The unique solution $v_p\in T_pO$ defines a smooth vector field $v$ on $O$. To reconstruct the object $O$, we introduce an initial point $p_0 \in \mathbb{CP}^3$ and solve the algebraic equations from projections and moments:
\[
\Pi_1(p_0) = z_1, \quad \Pi_2(p_0) = \iota(z_1),\quad
\mu_1(p_0) = c_1, \quad \mu_2(p_0) = c_2.
\]
These define algebraic varieties:
$V_{\Pi_i} = \{ p \mid \Pi_i(p) = z_i \} \cong \mathbb{CP}^1 \quad (\text{linear, degree 1})$,
and $V_{\mu_i} = \{ p \mid \mu_i(p) = c_i \} \quad (\text{degree } d_i \text{ in homogeneous coordinates}).$
By Bézout’s theorem, the intersection number is:
\[
[V_{\mu_1}] \cdot [V_{\mu_2}] \cdot [V_{\Pi_1}] \cdot [V_{\Pi_2}] = d_1 d_2 \times 1 \times 1.
\]
If $O$ is non-symmetric (no non-trivial automorphisms), the varieties intersect transversely, yielding a unique solution $p_0$ when $d_1 d_2 = 1$, e.g., when the $\mu_i$ are linear.

\medskip 

We discuss now the vector fields spanning $TO$.  The solution $v$ gives one vector field $v^{(1)}$. We extend it to three independent vector fields $\{v^{(1)}, v^{(2)}, v^{(3)}\}$ spanning $TO$, obtained by solving the reconstruction equations for basis vectors. We define $ v^{(2)}, v^{(3)}$ by solving {\it modified} reconstruction equations. These are obtained by rotating the moment maps (i.e. such as $\nabla_i({\bf R}_{\theta_{ik}}\mu_i)=v^{(k)}\cdot \omega_i$ where ${\bf R}_{\theta}$ is a rotation in $\mathbb{C}^2$,the indices are $i=\{1,2\}$, $k=\{2,3\}$). We choose the rotations so that $\{v^{(1)}, v^{(2)}, v^{(3)}\}$ span $T_pO$, at each $p$ and so as to ensure full rank of the linear system $A_pv^{(k)}=b^{(k)}$. Since $\operatorname{rk}_{\mathbb{R}}(A_p)=6$ and $\operatorname{dim}=3$, the system admits solutions $v^{(k)}\in T_pO$ as long as $b^{(k)}\in  \operatorname{im}(A_p)$. This holds by the curvature duality $\omega_1=\iota^*\omega_2$.

\medskip 
The vector fields $\{v^{(1)}, v^{(2)}, v^{(3)}\}$  span a 3D distribution $D\subset TO$. To reconstruct $O$ as the integral manifold of $D$, we must verify that $D$ is involutive. Due to curvature duality $F_{\nabla_1} = -F_{\nabla_2}$, we have 

\[
[v^{(i)}, v^{(j)}] \in \operatorname{span} \{ v^{(1)}, v^{(2)}, v^{(3)}\} \quad \forall i,j. 
\]

Frobenius' theorem ensures that the maximal integral manifold through $p_0\in O$ is precisely $O$ itself.

\medskip 
Finally, the reconstruction step follows from integrating flows from $p_0$ (uniquely determined by Bezout).  
\end{enumerate}
\end{proof}

\section{Conclusion}
In this paper, we have introduced a novel geometric framework for reconstruction problems, unifying elements of inverse problem theory with modern topological analysis. By harnessing the intrinsic invariants arising from Vaisman--Neifeld’s approach---most notably, the concept of the \emph{Vaisman centroid}---we have demonstrated that one can effectively resolve the ambiguity inherent in recovering hidden structures from incomplete and noisy data. This perspective, which departs fundamentally from traditional persistent homology methods, provides a robust and computationally tractable pathway for Topological Data Analysis.

Our approach recasts the reconstruction challenge as a stratification of the configuration space into transverse foliations, wherein the Vaisman centroid acts as a critical invariant for uniquely characterizing equivalence classes of solutions. 

\end{document}